\documentclass[a4paper,12pt]{article}
\usepackage{amsmath, amsthm}
\usepackage{mathtools}
\usepackage{amssymb}
\usepackage{amscd}
\usepackage{epic, eepic}
\usepackage{url}
\usepackage{color}
\usepackage[utf8]{inputenc} 
\usepackage{comment}
\usepackage{todonotes}
\usepackage{graphicx}
\usepackage{epstopdf}
\usepackage{enumerate}
\usepackage{array}
\usepackage{tabularx}
\usepackage{tikz}
\usepackage{enumitem}
\usepackage{tcolorbox}
\usepackage{pdflscape}
\usepackage{svg}
\usepackage{hyperref}

\usepackage{caption}
\usepackage{subcaption}

\newcommand{\cp}{\textrm{cp}}
\newcommand{\fp}{\textrm{fp}}
\newcommand{\fvs}{\textrm{fvs}}
\newcommand{\FVS}{\textrm{FVS}}

\newtheorem{theorem}{\bf Theorem}[section]
\newtheorem{lemma}[theorem]{\bf Lemma}
\newtheorem{cor}[theorem]{\bf Corollary}

\newtheorem{problem}[theorem]{\bf Problem}

\newtheorem{nota}[theorem]{\bf Notation}

\newtheorem{claim}[theorem]{\bf Claim}

\newtheorem{remark}[theorem]{\bf Remark}
\newtheorem{defi}[theorem]{\bf Definition}

\newtheorem{conjecture}[theorem]{\bf Conjecture}

\title{Jones' conjecture for Halin graphs and a bit more}

\date{}

\author{Pál Bärnkopf \thanks{Alfréd Rényi Institute of Mathematics, Budapest, Hungary. Partially supported by the Counting in Sparse Graphs Lendület Research Group. This project was partially supported by the ERC grant ERMiD.
E-mail: {\tt barpal@student.elte.hu}} 
\and  
Ervin Győri \thanks{Alfréd Rényi Institute of Mathematics, Budapest, Hungary, Partially supported by the National Research, Development and Innovation Office NKFIH, grants  K132696 and SNN135643, 	E-mail: {\tt gyori@renyi.hu}}
}
  
\begin{document}

\maketitle

\begin{abstract}
We prove Jones' famous conjecture for Halin graphs and a somewhat more general class of graphs, too. 
A based planar graph is a planar one that has a face adjacent to every other face. We confirm Jones' conjecture for based planar graphs. Namely, if a based planar graph does not contain $k+1$ vertex-disjoint cycles, then it suffices to delete $2k$ vertices to make it acyclic.

\end{abstract}

\section{Introduction}

A classical theorem of Erdős and Pósa \cite{EP} says that if a graph $G$ does not contain more than $k$ pairwise disjoint cycles, then we can delete $f(k)\le ck\log k$ vertices to make the graph acyclic. 
Jones' conjecture \cite{kloks} raises the same problem in planar graphs. It states that if a planar graph contains no more than $k$ pairwise disjoint cycles, then there are $2k$ vertices such that removing them from the graph gives a forest.  Jones' conjecture was proved for outerplanar graphs \cite{kloks} and subcubic planar graphs \cite{bonamy}. Currently, the best known general result is that if the maximum number of disjoint cycles in an arbitrary planar graph is $k$, then there are $3k$ vertices such that removing them from the graph we get a forest, as independently proved by Chen et al. \cite{chen}, Ma et al. \cite{ma} and Chappell et al. \cite{chappell}.

The main result of this paper is to prove the Jones' conjecture for planar graphs that have a face that is adjacent to every other face. Let us call these graphs \textit{based planar graphs}. (As a generalization of the based polyhedra introduced by Rademacher \cite{rademacher}.)

Throughout the paper, all graphs are simple, even if this is not indicated separately in the statement.

\begin{defi}
 A Halin graph is a planar graph constructed by connecting the leaves of a tree into a cycle.   
\end{defi}

\begin{defi}
    A based planar graph is a planar graph that has a face that is adjacent to every other face.
\end{defi}

\begin{remark}
Notice that a Halin graph is a based planar graph.
\end{remark}

\begin{defi}
    A cycle packing of a graph $G$ is a set of vertex disjoint cycles that appear in $G$ as subgraphs. We denote the maximum size of a cycle packing of $G$ by $\cp(G)$.
\end{defi}

\begin{defi}
    A face packing of a planar graph $G$ is a set of vertex disjoint faces that appear in $G$ as subgraphs. We denote the maximum size of a face packing of $G$ by $\fp(G)$.
\end{defi}

\begin{defi}
    A feedback vertex set of a graph $G$ is a set $S$ of vertices such that $G \setminus S$ is a forest. We denote an arbitrary minimum feedback vertex set of $G$ as $\FVS(G)$ and denote its size by $\fvs(G)$.
\end{defi}

\begin{theorem} \label{main}
    Every based planar graph $G$ satisfies $\fvs(G) \leq 2 \cdot \cp(G)$.
\end{theorem}

\begin{cor}
    Every Halin graph $G$ satisfies $\fvs(G) \leq 2 \cdot \cp(G)$.
\end{cor}

\section{Proof of Theorem \ref{main}}

From now on, the outer face of a based planar graph is adjacent to every other face. 

\begin{nota}
        In the following proof \textit{a good triangle} is a triangle $xyz$, for which $d(x)=d(y)=3$ so that the vertices $x$ and $y$ are on the outer face. A good triangle for a vertex $u$ is a good triangle that does not contain $u$.
\end{nota}

\begin{claim} \label{degree3}
    If the degree of each vertex in a based planar graph $G$ is at least three and the degree of each vertex on the outer face is $3$, then for an arbitrary vertex $u$ on the outer face, there is a good triangle $xyz$ in the graph $G$ for the vertex $u$. 
\end{claim}

\begin{proof}[Proof of Claim \ref{degree3}]
    The proof is done by induction on the number of vertices. The number of vertices of $G$ is at least 4 and if the number of vertices is 4, then the graph is $K_4$, for which the statement is obviously true.

    If we omit the edges at the border of the outer face, we get a forest. If this forest is a star, then our graph is a wheel, for which the statement holds. 
    
    If this forest is a tree, but it is not a star, then let us take a longest path in it. Consider one of the two penultimate vertices of the path (they are different because the graph is not a star). Every neighbor of that vertex, other than its neighbor on the path, is a leaf, because if not, then the path would not be a longest path. Thus the penultimate vertex has at least two leaf neighbors. Let us take two of them. The two leaf-edges divide the inner region into two subregions, the nonterminal vertices of the longest path are in one of these subregions. Now consider the other subregion. In this subregion every edge incident to the penultimate vertex is a leaf edge. (If not then either there is a longer path or the forest would not be connected.) So, every vertex of this subregion except the penultimate vertex is a leaf. Two of these leaves will be adjacent on the border of the outer face, so we have found a good triangle. We can do the same at the other end of the longest path. So, we found two vertex-disjoint good triangles. It follows that there is a good triangle for every vertex.

    Consider the case where the resulting forest consists of several components. Let us choose one of the trees as we want. Its leaves divide the border of the outer face into arcs. Let us choose another tree; its vertices are on one of the resulting arcs. We prove that we may assume that neither of the two trees has only two vertices. 
    
    Assume that one of these trees has only two vertices. (If there are two, take any of them.) Then these two vertices divide the cycle into two arcs. No multiple edges, so there are additional vertices (and trees) on both arcs. Instead of the previously chosen trees, choose the tree in each arc that covers a smallest section of this arc. If $F$ is such a tree then no other tree is associated to the arc covered by $F$. So, $F$ should have more than 2 vertices as we have seen above.
    
    Denote $a$ and $b$ the two leaves of the first tree, which are the endpoints of the associated arc that disjoint to the second tree, and denote $c$ and $d$ the two extreme leaves of the second tree on the arc so that the vertices follow each other in order $a,c,d,b$ at the border of the outer face. Denote $\mathcal{C}$ the arc $ab$ that contains the vertices $c$ and $d$. Delete the edges connected to $a$ and $c$ on the arc $ac$ contained in $\mathcal{C}$ (these edges are maybe the same) and delete the edges connected to $b$ and $d$ on the arc $bd$ contained in $\mathcal{C}$ (these edges are maybe the same) and connect with an edge the vertices $a$ and $b$ and the vertices $c$ and $d$. Thus, we get at least two components. We can use induction for two of them: $G_1$ that contains the vertices $a$ and $b$ and $G_2$ that contains the vertices $c$ and $d$.
    By induction, there is a good triangle for $a$ in $G_1$ and a good triangle for $c$ in $G_2$. Neither of the two triangles can use either of the two new edges, so they are vertex-disjoint good triangles in the original graph $G$. It follows that there is a good triangle for every vertex. (It is not necessary, but it is sufficient to find two disjoint triangles whose vertices on the outer face have degree $3$.)
\end{proof}

\begin{lemma} \label{triangle}
    If the degree of each vertex in a based planar graph $G$ is at least three, then for an arbitrary vertex $u$ on the outer face, there is a good triangle $xyz$ in the graph $G$ for the vertex $u$. 
\end{lemma}

\begin{proof}[Proof of Lemma \ref{triangle}]

    Let $k$ be the number of vertices on the outer face whose degree is greater than three. The proof is done by induction according to $k$ and within it by induction according to the number of vertices. (So, if the value of $k$ is smaller for another graph, or the value of $k$ is the same, but the number of vertices is smaller, then according to the induction assumption, the statement is true for the other graph.) If $k = 0$, then the statement is true because of Claim \ref{degree3}.

    If $k>0$, then let us consider a vertex $v$ on the outer face whose degree is greater than three. If the vertex $v$ is a cut vertex, then $G=G_1 \cup G_2$ such that $G_1 \cap G_2=\{v\}$ and $G_1 \setminus \{v\}$ and $G_2 \setminus \{v\}$ are disjoint. If the degree of the vertex $v$ in the graph $G_1$ or in the graph $G_2$ is at least three, then we are done. Since, we may assume that the degree of the vertex $v$ is at least three in the graph $G_1$. The graph $G_1$ has fewer vertices than $G$ and the value of $k$ cannot increase with the operation. So, there will be a good triangle for every vertex on the outer face of $G_1$, due to induction. Any of these triangles will also be good for the vertices of the graph $G_2 \setminus \{v\}$. Consider the case when the degree of the vertex $v$ in both graph $G_ 1$ and $G_ 2$ is two. Denote the neighbors of the vertex $v$ in the graph $G_1$ by $e$ and $f$ and in the graph $G_2$ by $g$ and $h$. Delete the vertex $v$ and connect the vertices $e$ and $g$ in addition the vertices $f$ and $h$ with one edge each. We can apply induction to the resulting graph. There will be a good triangle for each vertex and this triangle will be a good triangle in the original graph too, since we did not create a new triangle (and these triangles will be good for the vertex $v$, too). 
    
    Let us consider a vertex $v$ on the outer face whose degree is greater than three and that is not a cut vertex. Let us split the vertex $v$ into a vertex $v_1$ and a vertex $v_2$ ($v_1$ and $v_2$ should be adjacent). The two edges that are on the border of the outer face and run into the vertex $v$, after that let them run into the vertex $v_1$. The other edges that run into the vertex $v$, after that let them run into the vertex $v_2$. This operation maintains that the degree of each vertex is at least three (the degree of the vertices $v_1$ and $v_2$ is at least three and the degrees of the other vertices do not change). It is preserved that the graph is a planar graph and that every face is adjacent to the outer face. The value of $k$ has decreased for the graph, so we can apply the induction assumption to it. The vertex $v_1$ will not be included in a triangle. If the vertex $v_2$ is included in a good triangle for some vertex $u$, it corresponds to a good triangle in the original graph for $u$ in which the vertex $v$ is present instead of the vertex $v_2$. 
    
    The only problem is that the induction does not guarantee a good triangle for the vertex $v$.
    For this we have to stipulate that the vertex $v_2$ is not included in the triangle, but the vertex $v_2$ is not on the border of the outer face. Assume by contradiction that every good triangle that has two three-degrees vertices on the outer face contains the vertex $v$. There must be at least two such triangles, since for every vertex different from the vertex $v$ we proved by induction that there is a good triangle that does not contain that vertex. Let these two good triangles be $vab$ and $vcd$ such that they are located in the order $vabcdv$ on the border of the outer face. $d(a)=d(b)=d(c)=d(d)=3$, so neither $a$ nor $d$ can be the neighbor of $v$ on the border of the outer face. The vertices $a$ and $d$ divide the boundary of the outer face into two arcs. Let us take the one that contains the vertices $b$ and $c$. This arc, together with the vertices $a, d$, and $v$, defines a cycle. Delete all the vertices that are inside the cycle (in the planar embedding) and all internal points of the arc and join the vertices $a$ and $d$ with an edge. The resulting graph satisfies the conditions of the lemma: the degree of every vertex is at least three (because $a$ and $d$ are adjacent to their neighbors on the border of the outer face, to $v$ and to each other) and every face is adjacent to the outer face. Furthermore, the number of vertices on the outer face whose degree is greater than three cannot increase. So we can use the induction assumption. In the resulting graph, we find a good triangle that does not contain $v$, this is also a suitable triangle in the original graph, so we have reached a contradiction. Since it cannot happen that we do not find a good triangle for the vertex $v$, the induction works and the statement of the lemma is true.    
\end{proof}

\begin{proof}[Proof of Theorem \ref{main}]
    The proof is done by induction on $n$, the number of vertices. If $n \leq 3$, then the theorem is trivially true. 
    
\begin{remark}
    All subgraphs of a based planar graph are based planar graphs. We never select the outer face in the cycle packing. This does not narrow the meaning of cycle packing, since we only delete one trivial case, but it makes the proofs technically simpler.
\end{remark}

    If there is a vertex $u$ in the graph $G$ of which degree is smaller than $2$, then we delete one of those vertices. Let the resulting graph be $G'$. $\cp(G)=\cp(G')$ and $\fvs(G)=\fvs(G')$, since $u$ is not in any cycle. The statement is true for the graph $G'$ by induction, so it is also true for the graph $G$.

    If there is a triangle $xyz$ in the graph $G$ such that $d(x)=2$, then let $G'$ be the graph obtained by omitting the vertices $y$ and $z$ from the graph $G$. $\cp(G') \leq \cp(G)-1$ since let we take a cycle packing in the graph $G'$, in the graph $G$ we can also add the triangle $xyz$ to this cycle packing. Thus for every cycle packing in the graph $G'$ we find a larger cycle packing in the graph $G$, so the maximum size of a cycle packing of $G$ is necessarily bigger than the maximum size of a cycle packing of $G'$. Let $\cp(G)=k$ then $\cp(G') \leq k-1$, so $\fvs(G') \leq 2k-2$ by induction. Let $F'$ be an arbitrary minimum feedback vertex set of the graph $G'$, then $F=F' \cup \{y,z\}$ is a feedback vertex set of the graph $G$ because the difference between the graph $G \setminus F$ and the graph $G' \setminus F'$ is only an isolated vertex. Thus, $\fvs(G) \leq \fvs(G')+2 \leq 2k$.

    If there is a vertex $u$ of degree two in the graph $G$, which is not in a triangle, then let $G'$ be the graph obtained by omitting the vertex $u$ from the graph $G$ and connecting the two neighbors of the vertex $u$ with an edge. This graph is trivially still a based planar graph. Any minimum feedback vertex set in this graph is also a minimum feedback vertex set in the original graph. If a minimum feedback vertex set in $G$ does not contain $u$, then it is a minimum feedback vertex set in $G'$, too. If a minimum feedback vertex set in $G$ contains $u$, then let us delete $u$ from it and add a neighbor of $u$. So, we get a minimum feedback vertex set with the same size in $G'$. Furthermore, similar is true for cycle packings: they are essentially identical in the two graphs. So $\cp(G)=\cp(G')$ and $\fvs(G)=\fvs(G')$, thus, by induction, we are done. (The statement is true for the graph $G'$, so it is also true for the graph $G$.)

    Now, we may assume that the degree of each vertex in the graph is at least three. Then, according to Lemma \ref{triangle}, there is a triangle $xyz$, for which $d(x)=d(y)=3$ such that the vertices $x$ and $y$ are on the outer face. Let $G'$ be the graph obtained by deleting the vertices $y$ and $z$ from the graph $G$. $\cp(G')<\cp(G)$, since the vertex $x$ cannot be included in any cycle in the graph $G'$, so we can add the triangle $xyz$ to an arbitrary cycle packing of the graph $G'$ and thus get a bigger cycle packing. Then let us take a minimum feedback vertex set of the graph $G'$, adding the vertices $y$ and $z$ to this, we get a feedback vertex set of the graph $G$, whose size is at most $2 \cp(G')+2 \leq 2 \cdot \cp(G)$.
\end{proof}

The proof works in the same way, if instead of cycle packing, we considered face packing, so the following stronger statement also holds.

\begin{remark}
    Every based planar graph $G$ satisfies $\fvs(G) \leq 2 \cdot \fp(G)$.
\end{remark}

\section{Concluding remarks}

The proof takes advantage of the fact that there is a face that is adjacent to every other face. The question arises naturally whether the statement is true even if there is a cycle in the graph which has at least one edge in common with every face?

\begin{conjecture}
    Let $G$ be a planar graph in which there is a cycle which has at least one edge in common with every face. Prove $\fvs(G) \leq 2 \cdot \cp(G)$.
\end{conjecture}

Another similar approach might be the special case of hamiltonian graphs.

\begin{problem}
    Let $G$ be a hamiltonian planar graph. Prove $\fvs(G) \leq 2 \cdot \cp(G)$.
\end{problem}



\begin{thebibliography}{99}

\bibitem{bonamy} M. Bonamy, F. Dross, T. Masařík, W. Nadara, M. Pilipczuk, M. Pilipczuk, Jones' Conjecture in subcubic graphs, Electronic Journal of Combinatorics, 2021.

\bibitem{chappell} G.G. Chappell, J. Gimbel, C. Hartman, On cycle packings and feedback vertex sets, Contributions to Discrete Mathematics 9.2, (2014)

\bibitem{chen} H.B. Chen, H.L. Fu, C.H. Shih, Feedback vertex set on planar graphs, Taiwanese Journal of Mathematics 16.6 (2012), 2077-2082.

\bibitem{EP} P. Erdős, L. Pósa, On independent circuits contained in a graph. Canadian Journal of Mathematics, 17(1965), 347–352.

\bibitem{kloks} T. Kloks, C.M. Lee, J. Liu, New algorithms for k-face cover, k-feedback vertex set, and k-disjoint cycles on plane and planar graphs. In International Workshop on Graph-Theoretic Concepts in Computer Science Berlin, Heidelberg: Springer Berlin Heidelberg (2002), 282-295.

\bibitem{ma} J. Ma, X. Yu, W. Zang, Approximate min-max relations on plane graphs, Journal of Combinatorial Optimization 26.1 (2013), 127-134.

\bibitem{rademacher} H. Rademacher, On the number of certain types of polyhedra, Illinois Journal of Mathematics 9.3 (1965), 361-380.


\end{thebibliography}

\section*{Statements and declarations}
The authors declare that they has no conflict of interest.
Data sharing not applicable to this article as no datasets were generated or analysed during the current study.

\end{document}